\documentclass[12pt,reqno]{amsart}
\usepackage{enumerate, latexsym, amsmath, amsfonts, amssymb, amsthm, color}
\usepackage{psfrag,graphicx, colordvi}
\newtheorem{theo}{Theorem}

\def\Z{\mathbb Z}
\def\N{\mathbb N}

\def\l{\left}
\def\r{\right}
\def\bg{\bigg}
\def\({\bg(}
\def\){\bg)}

\def\f{\frac}
\def\poq#1#2{(#1;q)_#2}

\def\bi{\binom}

\theoremstyle{plain}

\newtheorem{lemma}{Lemma}

\theoremstyle{definition}

\theoremstyle{remark}
\newtheorem{remark}{Remark}

\makeatletter \@addtoreset{equation}{section}
\@addtoreset{theo}{section} \makeatother

\begin{document}
\baselineskip=17pt
\medskip

\hbox{Proc. Amer. Math. Soc. 147(2019), no.\,5, 1953--1961.}
\medskip

\title
[On $q$-analogues of some series for $\pi$ and $\pi^2$]
{On $q$-analogues of some series for $\pi$ and $\pi^2$}

\author
[Qing-Hu Hou, Christian Krattenthaler and Zhi-Wei Sun]
{Qing-Hu Hou$^\dagger$, Christian Krattenthaler$^\ddagger$ and Zhi-Wei Sun$^\star$}

 \address {(Qing-Hu Hou) School of Mathematics, Tianjin University,
  Tianjin 300350, People's Republic of China}

\email{{\tt qh\_hou@tju.edu.cn}
\newline\indent
{\it Homepage}: {\tt http://cam.tju.edu.cn/\lower0.5ex\hbox{\~{}}hou}}

\address {(Christian Krattenthaler) Fakult\"at f\"ur Mathematik, Universit\"at Wien,
Oskar-Morgenstern-Platz~1, A-1090 Vienna, Austria}

\email{{\tt christian.krattenthaler@univie.ac.at}
\newline\indent
{\it Homepage}: {\tt http://www.mat.univie.ac.at/\lower0.5ex\hbox{\~{}}kratt}}

\address {(Zhi-Wei Sun) Department of Mathematics, Nanjing
University, Nanjing 210093, People's Republic of China}

\email{{\tt zwsun@nju.edu.cn}
\newline\indent
{\it Homepage}: {\tt http://maths.nju.edu.cn/\lower0.5ex\hbox{\~{}}zwsun}}


\thanks{$^\dagger$ Supported by the
  National Natural Science Foundation of China (grant 11771330).\newline
\indent
$^\ddagger$ Partially supported by the Austrian
Science Foundation FWF (grant S50-N15)
in the framework of the Special Research Program
``Algorithmic and Enumerative Combinatorics".
\newline
\indent
$^\star$ Corresponding author. Supported by the
  National Natural Science Foundation of China (grant 11571162).}

\begin{abstract} We obtain a new $q$-analogue of the classical Leibniz series $\sum_{k=0}^\infty(-1)^k/(2k+1)=\pi/4$, namely
\begin{equation*}
\sum_{k=0}^\infty\frac{(-1)^kq^{k(k+3)/2}}{1-q^{2k+1}}=\frac{(q^2;q^2)_{\infty}(q^8;q^8)_{\infty}}{(q;q^2)_{\infty}(q^4;q^8)_{\infty}},
\end{equation*}
where $q$ is a complex number with $|q|<1$.
We also show that the Zeilberger-type series
$\sum_{k=1}^\infty(3k-1)16^k/(k\binom{2k}k)^3=\pi^2/2$
has two $q$-analogues with $|q|<1$, one of which is
$$\sum_{n=0}^\infty q^{n(n+1)/2} \frac {1-q^{3n+2}} {1-q}
\cdot\frac{(q;q)_n^3 (-q;q)_n}{(q^3;q^2)_{n}^3} = (1-q)^2
\frac{(q^2;q^2)^4_\infty}{(q;q^2)^4_\infty}.$$
\end{abstract}

\maketitle

\section{Introduction}
Let $q$ be a complex number with $|q|<1$.
As usual, for $n\in\mathbb N=\{0,1,2,\ldots\}$ and
a complex number $a$, we define the $q$-shifted factorial by
$$(a;q)_n=\prod_{k=0}^{n-1}(1-aq^k).$$
(An empty product is considered to take the value $1$, and thus $(a;q)_0=1$.)
We also adopt the standard notion
$$(a;q)_\infty = \lim_{n\to\infty}(a;q)_n=\prod_{k=0}^{\infty} (1-a q^k).$$
By the definition of the $q$-Gamma function \cite[p.~20]{Gas04}, we have
\[
\frac{(q^2;q^2)_\infty}{(q;q^2)_\infty} = \Gamma_{q^2}\left( \frac{1}{2} \right) (1-q^2)^{-1/2}.
\]
Therefore,
\begin{equation}\label{pi}
\lim_{q \to 1} (1-q) \frac{(q^2;q^2)^2_\infty}{(q;q^2)^2_\infty} = \Gamma \left( \frac{1}{2} \right)^2 \lim_{q \to 1} \frac{1-q}{1-q^2} = \frac{\pi}{2}
\end{equation}
and
\[
\lim_{q \to 1} (1-q^2) \frac{(q^4;q^4)^2_\infty}{(q^2;q^4)^2_\infty} = \frac{\pi}{2}.
\]
In view of this, Ramanujan's formula
\begin{equation}\label{pi1}
\sum_{k=0}^\infty\f{(-q)^k}{1-q^{2k+1}}=\f{(q^4;q^4)_{\infty}^2}{(q^2;q^4)_{\infty}^2}\ \ (|q|<1)
\end{equation}
(equivalent to Example (iv) in \cite[p.\,139]{B91}) can be viewed as a $q$-analogue of
Leibniz's identity
\begin{equation}\label{Leib}\sum_{k=0}^\infty\f{(-1)^k}{2k+1}=\f{\pi}4.
\end{equation}

Guo and Liu \cite{GL} used the WZ method to deduce the identities
\begin{equation} \label{eq:GL1}
\sum_{n=0}^\infty q^{n^2} \frac {1-q^{6n+1}} {1-q} \cdot\frac{(q;q^2)_n^2
  (q^2;q^4)_n}{(q^4;q^4)_n^3} = \frac{(1+q) (q^2;q^4)_\infty
  (q^6;q^4)_\infty}{(q^4;q^4)_\infty^2}
\end{equation}
and
\begin{equation} \label{eq:GL2}
\sum_{n=0}^\infty (-1)^nq^{3n^2} \frac {1-q^{6n+1}} {1-q}
\cdot\frac{(q;q^2)_n^3}{(q^4;q^4)_n^3} = \frac{(q^3;q^4)_\infty
  (q^5;q^4)_\infty}{(q^4;q^4)_\infty^2}
\end{equation}
with $|q|<1$, which are $q$-analogues of Ramanujan's formulas \cite[p.~352]{Ber94}
\[
\sum_{n=0}^\infty (6n+1)\frac{(1/2)_n^3}{n!^3 4^n} =\sum_{n=0}^\infty(6n+1)\f{\bi{2n}n^3}{256^n}= \frac{4}{\pi}
\]
and
\[\sum_{n=0}^\infty (6n+1)(-1)^n\frac{(1/2)_n^3}{n!^3 8^n} =\sum_{n=0}^\infty(6n+1)\f{\bi{2n}n^3}{(-512)^n}= \frac{2\sqrt2}{\pi},
\]
where $(a)_n = \prod_{k=0}^{n-1}(a+k)$ is the Pochhammer symbol. Note that
$$\f{\l(1/2\r)_n}{n!}=(-1)^n\bi{-1/2}n=\f{\bi{2n}n}{4^n}\quad\mbox{for all}\ n\in\N.$$

Quite recently, Sun \cite{Sun18} provided $q$-analogues of Euler's
classical formulas $\zeta(2)=\pi^2/6$ and $\zeta(4)=\pi^4/90$.

In 1993, Zeilberger \cite{Zeil93} used the WZ method to show that
\[
\sum_{n=0}^\infty (21n+13) \frac{n!^6}{8 (2n+1)!^3} =\sum_{k=1}^\infty\f{21k-8}{k^3\bi{2k}k^3}= \frac{\pi^2}{6}.
\]
A complicated $q$-analogue of the identity $\sum_{k=1}^\infty(21k-8)/(k^3\binom{2k}k)^3=\zeta(2)$ was given by 
Hessami Pilehrood and Hessami Pilehrood \cite{HP} in 2011.
Following Zeilberger's work, in 2008 Guillera \cite[Identity~1]{Gui08} employed the WZ method to obtain the Zeilberger-type series
\begin{equation}\label{o-id}
\sum_{k=1}^\infty  \frac{(3k-1)16^k}{k^3 {\binom {2k} k}^3} = \frac{\pi^2}{2}.
\end{equation}

In this paper we study $q$-analogues of the identities \eqref{Leib}
and \eqref{o-id}.

Now we state our main results.

\begin{theo}\label{Th1.1} For $|q|<1$ we have
\begin{equation}\label{pi2}
\sum_{k=0}^\infty\f{(-1)^kq^{k(k+3)/2}}{1-q^{2k+1}}=\f{(q^2;q^2)_{\infty}(q^8;q^8)_{\infty}}{(q;q^2)_{\infty}(q^4;q^8)_{\infty}}.
\end{equation}
\end{theo}

The above identity gives a new $q$-analogue of Leibniz's identity \eqref{Leib}
since
\[
\lim_{q \to 1} (1-q^2) \f{(q^2;q^2)_{\infty}(q^8;q^8)_{\infty}}{(q;q^2)_{\infty}(q^4;q^8)_{\infty}}
=\lim_{q\to1}\Gamma_{q^2}\l(\f12\r)\Gamma_{q^8}\l(\f12\r)=\Gamma\l(\f12\r)^2=\pi.
\]

\begin{theo} \label{th-main}
For $|q|<1$ we have
\begin{equation}\label{q2}
 \sum_{n=0}^\infty q^{2n(n+1)} (1 + q^{2n+2} - 2q^{4n+3}) \frac{(q^2;q^2)_{n}^3}{(q;q^2)_{n+1}^3 (-1;q)_{2n+3}} = \f12 \sum_{n=0}^\infty \frac{q^{2n}}{(1-q^{2n+1})^2}
\end{equation}
and
\begin{equation}\label{q-id}
\sum_{n=0}^\infty q^{n(n+1)/2} \frac {1-q^{3n+2}} {1-q}
\cdot\frac{(q;q)_n^3 (-q;q)_n}{(q^3;q^2)_{n}^3} = (1-q)^2
\frac{(q^2;q^2)^4_\infty}{(q;q^2)^4_\infty}.
\end{equation}
\end{theo}
Multiplying both sides of \eqref{q2} by $(1-q)^2$ and then letting $q \to 1$, we obtain
\[
\frac{1}{4} \sum_{n=0}^\infty (3n+2) \frac{2^{4n} n!^6}{(2n+1)!^3} = \f12\sum_{n=0}^\infty \frac{1}{(2n+1)^2} =
\f12\l(1-\f14\r)\zeta(2)= \frac{\pi^2}{16},
\]
which is equivalent to \eqref{o-id}.
In view of \eqref{pi}
and the fact that
\[
\lim_{q \to 1} q^{n(n+1)/2} \frac {1-q^{3n+2}} {1-q}
\cdot\frac{(q;q)_n^3 (-q;q)_n}{(q^3;q^2)_{n}^3} =  \frac{(3n+2)16^{n+1}}{2
  (n+1)^3 {\binom {2n+2} {n+1}}^3},
\]
the identity \eqref{q-id} is also a $q$-analogue of \eqref{o-id}. The expansions of both sides of \eqref{q-id} are
$$1+2q-q^2+3q^4-6q^5+3q^6+8q^7-16q^8+8q^9+10q^{10}+\cdots.$$

In \cite[Conjecture~1.4]{Sun11}, Sun presented several conjectural identities similar to Zeilberger-type series
(one of which is $\sum_{k=1}^\infty(10k-3)8^k/(k^3\bi{2k}k^2\bi{3k}k)=\pi^2/2$), but we could not find $q$-analogues of them.

We are going to show Theorems \ref{Th1.1} and \ref{th-main} in
Sections~\ref{sec:2} and~\ref{sec:3} respectively. Finally, in
Section~\ref{sec:4}, we give alternative proofs for \eqref{eq:GL1}
and \eqref{eq:GL2}.

\section{Proof of Theorem~\ref{Th1.1}}
\label{sec:2}

As usual, for $x\in\Z$ we let $T_x$ denote the triangular number $x(x+1)/2$.

\begin{lemma}\label{Lem2.1} Let $n\in\N$, and define
$$t_2(n):=|\{(x,y)\in\N^2:\ T_x+4T_y=n\}|.$$
Then
\begin{equation}\label{t2'}t_2(n)=
\underset{d<\sqrt{8n+5}}{\sum_{d\mid 8n+5}}(-1)^{(d-1)/2}.
\end{equation}
\end{lemma}
\begin{proof}
By Theorem 3.2.1 of \cite[p.~56]{B06}, for any positive integer $m$ we have
 \begin{equation*}r_2(m)=4\sum_{2\nmid d\mid m}(-1)^{(d-1)/2},
\end{equation*}
where $r_2(m):=|\{(x,y)\in\Z^2:\ x^2+y^2=m\}|$. Observe that
\begin{align*}t_2(n)=&|\{(x,y)\in\N^2:\ (2x+1)^2+4(2y+1)^2=8n+5\}|
\\=&\f14|\{(x,y)\in\Z^2:\ x^2+(2y)^2=8n+5\}|=\f18r_2(8n+5)
\\=&\f12\sum_{d\mid 8n+5}(-1)^{(d-1)/2}=
\underset{d<\sqrt{8n+5}}{\sum_{d\mid 8n+5}}\f{(-1)^{(d-1)/2}+(-1)^{((8n+5)/d-1)/2}}2
\\=&\underset{d<\sqrt{8n+5}}{\sum_{d\mid 8n+5}}(-1)^{(d-1)/2}.
\end{align*}
This proves \eqref{t2'}.
\end{proof}

As usual, for $|q|<1$ we define
$$\psi(q):=\sum_{n=0}^\infty q^{T_n}.$$
By a known formula of Gau{\ss} (cf. (1.3.14) of \cite[p.~11]{B06}),
\begin{equation}\label{Gau}\psi(q)=\f{(q^2;q^2)_{\infty}}{(q;q^2)_{\infty}}.
\end{equation}

\begin{proof}[First proof of \eqref{pi2}]
Let $L$ and $R$ denote the left-hand side
and the right-hand side of \eqref{pi2}, respectively.
In view of Gau{\ss}' identity \eqref{Gau} and \eqref{t2'}, we have
$$R=\psi(q)\psi(q^4)=\sum_{n=0}^\infty t_2(n)q^n=\sum_{n=0}^\infty
\Bigg(\underset{d<\sqrt{8n+5}}{\sum_{d\mid 8n+5}}(-1)^{(d-1)/2}\Bigg)q^n.$$
On the other hand,
\begin{align*}L=&\sum_{k=0}^\infty(-1)^k\sum_{m=0}^\infty q^{k(k+3)/2+(2k+1)m}
\\=&\sum_{k=0}^\infty\sum_{m=0}^\infty(-1)^k q^{((2k+1)(2k+1+4(2m+1))-5)/8}
=\sum_{n=0}^\infty\underset{d<\sqrt{8n+5}}{\sum_{d\mid 8n+5}}(-1)^{(d-1)/2}q^n.
\end{align*}
Therefore \eqref{pi2} is valid.
\end{proof}

\begin{proof}[Second proof of \eqref{pi2}]
Recall the standard basic hypergeometric notation
$${}_r\phi_s\!\left[\begin{matrix} a_1,\dots,a_r\\
b_1,\dots,b_s\end{matrix}; q,z\right]
=\sum _{\ell=0} ^{\infty}\frac {\poq{a_1}{\ell}\cdots\poq{a_r}{\ell}}
{\poq{q}{\ell}\poq{b_1}{\ell}\cdots\poq{b_s}{\ell}}
\left((-1)^\ell q^{\binom\ell2}\right)^{s-r+1}z^\ell$$
(with $\poq{a_1}{\ell}\cdots\poq{a_r}{\ell}$ often
abbreviated as $(a_1,\ldots,a_r;q)_\ell$).
We start with the transformation formula
(cf.\ \cite[Eq.~(3.10.4)]{Gas04})
\begin{multline*}
{} _{10} \phi _{9} \! \left [             \begin{matrix} \let \over /
   \def\frac#1#2{#1 / #2} a, {\sqrt{a}} q, - {\sqrt{a}} q  , b,
   x, -x, y, -y, \\ \let \over / \def\frac#1#2{#1 / #2}
   {\sqrt{a}}, -{\sqrt{a}}, {{a q}\over b}, {{a q}\over x}, -{{a q}\over
   x}, {{a q}\over y}, -{{a q}\over y}, \end{matrix}\right.\\
\kern3cm
\left.
\begin{matrix} -{q^{-n}}, {q^{-n}}\\
- a {q^{n+1}} ,
   a {q^{n+1}}\end{matrix} ;q, {\displaystyle -{\frac {{a^3} {q^{2n+3}}}
   {b {x^2} {y^2}}}} \right ]
\\
=
{\frac{({\let \over / a^2 q^2, {{a^2 q^2}\over {x^2 y^2}}};
    q^2) _{n}} {({\let \over / {{a^2 q^2}\over {x^2}},
    {{a^2 q^2}\over {y^2}}}; q^2) _{n}}}\,
  {} _{5} \phi _{4} \! \left [ \begin{matrix} \let \over / {q^{-2 n}}, x^2, y^2,
    -{{a q}\over b}, -{{a q^2}\over b}\\ \let \over / {{x^2 y^2}\over
    {a^2 {q^{2 n}}}}, {{a^2 q^2}\over {b^2}}, - a q
    , - a q^2  \end{matrix} ;q^2, {\displaystyle q^2} \right
    ]  ,
\end{multline*}
where $n$ is a nonnegative integer. In this identity, set
$a=q$, $x=y=\sqrt q$, and let $n\to\infty$. In this way, we obtain
\begin{multline*}
{} _{4} \phi _{5} \! \left [             \begin{matrix} \let \over / \def\frac#1#2{#1
   / #2} q, b, {\sqrt{q}}, -{\sqrt{q}}\\ \let \over / \def\frac#1#2{#1 / #2}
   {{{q^2}}\over b}, {q^{{3\over 2}}}, -{q^{{3\over 2}}}, 0, 0\end{matrix} ;q,
   {\displaystyle {\frac {{q^4}} b}} \right ]
   \\=
  {\frac {
      ({\let \over / \def\frac#1#2{#1 / #2} {q^2}}; {q^2}) _{\infty} \,
      ({\let \over / \def\frac#1#2{#1 / #2} {q^4}}; {q^2}) _{\infty} }
    {{{({\let \over / \def\frac#1#2{#1 / #2} {q^3}}; {q^2}) _{\infty}
        }^2}}}
{} _{4} \phi _{3} \! \left [             \begin{matrix} \let \over /
       \def\frac#1#2{#1 / #2} q, -{{{q^2}}\over b}, -{{{q^3}}\over b}, q\\
       \let \over / \def\frac#1#2{#1 / #2} -{q^3}, -{q^2}, {{{q^4}}\over
       {{b^2}}}\end{matrix} ;{q^2}, {\displaystyle {q^2}} \right ] .
\end{multline*}
By performing the limit as $b\to0$, the above transformation formula is reduced
to
$$
{} _{3} \phi _{3} \! \left [             \begin{matrix} \let \over / \def\frac#1#2{#1
   / #2} q, {\sqrt{q}}, -{\sqrt{q}}\\ \let \over / \def\frac#1#2{#1 / #2}
   -{q^{{3\over 2}}}, {q^{{3\over 2}}}, 0\end{matrix} ;q, {\displaystyle {q^2}}
   \right ] = {\frac {
      ({\let \over / \def\frac#1#2{#1 / #2} {q^2}}; {q^2}) _{\infty} \,
      ({\let \over / \def\frac#1#2{#1 / #2} {q^4}}; {q^2}) _{\infty} }
    {{{({\let \over / \def\frac#1#2{#1 / #2} {q^3}}; {q^2}) _{\infty}
        }^2}}}\,
{} _{2} \phi _{2} \! \left [             \begin{matrix} \let \over /
       \def\frac#1#2{#1 / #2} q, q\\ \let \over / \def\frac#1#2{#1 / #2}
       -{q^2}, -{q^3}\end{matrix} ;{q^2}, {\displaystyle {q^3}} \right ] .
$$
If the left-hand side is written out explicitly, we see that it agrees
with the left-hand side in \eqref{pi2} up to a multiplicative
factor of $1-q$. On the other hand,
the $_2\phi_2$-series on the right-hand side can be evaluated by
means of the summation formula (cf.\ \cite[Ex.~1.19(i); Appendix
  (II.10)]{Gas04})
$$
{}_2\phi _2\!\left [ \begin{matrix} \let\over/ a,{q\over a}\\ \let\over/  -q,b\end{matrix} ;q,-b\right ] =
  { \frac {{(\let\over/ a b,{{b q}\over a};{q^2})}_{\infty}}  {{(\let\over/ b;q)}_{\infty}}} .
$$
Thus, we arrive at
$$
\sum_{k=0}^\infty\f{(-1)^kq^{k(k+3)/2}\,(1-q)}{1-q^{2k+1}}
 = {\frac {({\let \over / \def\frac#1#2{#1 / #2} {q^2}}; {q^2}) _{\infty}\,
({\let \over / \def\frac#1#2{#1 / #2} {q^4}}; {q^2}) _{\infty}
       \,{{({\let \over / \def\frac#1#2{#1 / #2} -{q^4}}; {q^4}) _{\infty}^2 }
        } }
      {({\let \over / \def\frac#1#2{#1 / #2} -{q^3}}; {q^2}) _{\infty} \,
      {{({\let \over / \def\frac#1#2{#1 / #2} {q^3}}; {q^2}) _{\infty} ^2}}}},
$$
which is indeed equivalent to \eqref{pi2}.
\end{proof}

\section{Proof of Theorem~\ref{th-main}}
\label{sec:3}

\begin{proof}[Proof of \eqref{q2}]
We construct a $q$-analogue of the WZ pair given by
Guillera \cite[Identity~1]{Gui08}.

Recall that a pair of bivariate functions $(F(n,k), G(n,k))$ is called a {\it WZ pair} \cite[Chapter~7]{PWZ96} if
\[
F(n+1,k) - F(n,k) = G(n,k+1)- G(n,k).
\]
It was shown (cf.\ \cite{Am97}) that
\begin{equation} \label{WZ-id}
\sum_{n=0}^\infty G(n,0) = \lim_{k \to \infty} \sum_{n=0}^\infty G(n,k) + \sum_{k=0}^\infty F(0,k).
\end{equation}

We make the following construction.
Let
\[
F_q(n,k) = 4 \cdot \frac {1-q^{2n}} {1-q} \cdot B_q(n,k)
\]
and
\[
G_q(n,k) = \frac{4(1 + q^{2n+1} -  2 q^{4n+2k+1})}{(1 - q)(1+q^{2n}) (1 + q^{2 n+1})} B_q(n,k),
\]
where
\[
B_q(n,k) = \frac{(q;q^2)_k^2 (q;q^2)_n^3 }{ (q^{2n+2};q^2)_k^2 (q^2;q^2)_n^3 (-1;q)_{2n}} q^{2n^2+4nk}.
\]
We can extend the definition of $B_q(n,k)$ from
nonnegative integers $n,k$ to any real numbers $n,k$ by defining
\[
(a;q)_n = \frac{(a;q)_\infty}{(aq^n;q)_\infty}.
\]

Let $a$ be a positive real number. It is straightforward to check that\break
$(F_q(n+a,k),G_q(n+a,k))$ is a WZ pair. Observing that $B_q(n,k)$
contains the factor $q^{4nk}$, we get
\[
\lim_{k \to \infty} \sum_{n=0}^\infty G_q(n+a,k) = 0.
\]
Thus,
\begin{equation} \label{q-WZ}
\sum_{n=0}^\infty G_q(n+a,0) = \sum_{n=0}^\infty F_q(a,k).
\end{equation}
Setting $a=1/2$ and noting that
\[
(q;q^2)_{n+1/2} = \frac{(q;q^2)_\infty}{(q \cdot q^{2n+1};q^2)_\infty}
= \frac{(q;q^2)_\infty (q^2;q^2)_n}{(q^2;q^2)_\infty}
\]
and
\[
(q^2;q^2)_{n+1/2} = \frac{(q^2;q^2)_\infty}{(q^2 \cdot q^{2n+1};q^2)_\infty}
= \frac{(q^2;q^2)_\infty (q;q^2)_{n+1}}{(q;q^2)_\infty}
\]
for any $n\in\N$, we infer that
\[
F_q \left( \frac{1}{2}, k \right) = \frac{2 q^{1/2} }{1-q}\cdot \frac{q^{2k}}{(1-q^{2k+1})^2} \cdot\frac{(q;q^2)_\infty^6}{(q^2;q^2)_\infty^6},
\]
and
\begin{align*}
G_q \left( n+\frac{1}{2}, 0 \right) = &\frac{4 q^{2n^2+2n+1/{2}} (1 + q^{2n+2} - 2q^{4n+3})}{(1-q)} \\
&\times \frac{ (q; q^2)_\infty^3 (q^{2n+3}; q^2)_\infty^3 }{(-1;q)_{2n+3} (q^{2n+2};q^2)_\infty^3 (q^2; q^2)_\infty^3}.
\end{align*}
After cancelling the common factors, we arrive at \eqref{q2}.
\end{proof}

\begin{remark}
Guillera \cite{Gui08} obtained the identity \eqref{o-id} via the WZ pair
\[
F(n,k) = 8n B(n,k), \quad G(n,k) = (6n+4k+1) B(n,k),
\]
where
\[
B(n,k) =  \frac{(2k)!^2 (2n)!^3}{2^{8n+4k}(n+k)!^2 k!^2 n!^4}.
\]
Since
\[
\lim_{q \to 1} F_q(n,k) = F(n,k)\quad\mbox{and}\quad \lim_{q \to 1} G_q(n,k) = G(n,k),
\]
the pair $(F_q(n,k), G_q(n,k))$ is indeed
a $q$-analogue of the pair $(F(n,k),\break G(n,k))$.
\end{remark}

\begin{proof}[Proof of \eqref{q-id}]
We start from the quadratic transformation formula (cf.\break \cite[Eq.~(3.8.13)]{Gas04})
\begin{multline} \sum_{n=0}^\infty \frac {(1-aq^{3n})\,(a,d,aq/d;q^2)_n\,
(b,c,aq/bc;q)_n}
{(1-a)\,(q,aq/d,d;q)_n\,(aq^2/b,aq^2/c,bcq;q^2)_n}q^n
\\=
\frac {(aq^2,bq,cq,aq^2/bc;q^2)_\infty}
{(q,aq^2/b,aq^2/c,bcq;q^2)_\infty}
   {}_3\phi _2\!\left [ \begin{matrix} b,c,aq/bc\\ dq,aq^2/d
\end{matrix} ;q^2,q^2\right ].
\label{CK1}
\end{multline}
We set $a=q^2$ and $b=c=q$. This yields
\begin{multline} \sum_{n=0}^\infty \frac {(1-q^{3n+2})\,(q^2,d,q^3/d;q^2)_n\,
(q,q,q;q)_n}
{(1-q^2)\,(q,q^3/d,d;q)_n\,(q^3,q^3,q^3;q^2)_n}q^n
\\=
\frac {(q^4,q^2,q^2,q^2;q^2)_\infty}
{(q,q^3,q^3,q^3;q^2)_\infty}
   {}_3\phi _2\!\left [ \begin{matrix} q,q,q\\ dq,q^4/d
\end{matrix} ;q^2,q^2\right ].
\label{CK2}
\end{multline}
At this point, we observe that the limit $d\to0$ applied to the
left-hand side of \eqref{CK2} produces exactly the left-hand side of
\eqref{q-id}.
Thus, it remains to show the limit $d\to0$ applied to the right-hand
side of \eqref{CK2} yields the right-hand side of \eqref{q-id}.

In order to see this, we take recourse to the transformation formula
(cf.\ \cite[Eq.~(3.3.1) or Appendix (III.34)]{Gas04})
\begin{multline*}
{} _{3} \phi _{2} \! \left [ \begin{matrix} \let \over / A, B, C\\ \let \over / D,
   E\end{matrix} ;q, {\displaystyle q} \right ] =
  {\frac {(\let \over / {{B C q}\over E}, {q\over E} ;q) _\infty}  {(\let
     \over / {{C q}\over E}, {{B q}\over E} ;q) _\infty} }
    {} _{3} \phi _{2} \! \left [ \begin{matrix} \let \over / {D\over A}, B, C\\ \let
     \over / D, {{B C q}\over E}\end{matrix} ;q, {\displaystyle
{\frac{A q} E}}
     \right ]\\
 - {\frac {(\let \over / {q\over E}, A, B, C, {{D q}\over E} ;q)
     _\infty}  {(\let \over / {{C q}\over E}, {{B q}\over E}, D, {E\over
     q}, {{A q}\over E} ;q) _\infty} }
    {} _{3} \phi _{2} \! \left [ \begin{matrix} \let \over / {{C q}\over E}, {{B
     q}\over E}, {{A q}\over E}\\ \let \over / {{D q}\over E}, {{{q^2}}\over
     E}\end{matrix} ;q, {\displaystyle q} \right ].
\end{multline*}
Here we replace $q$ by $q^2$ and set $A=B=C=q$, $D=dq$, and $E=q^4/d$,
to obtain
\begin{multline*}
   {}_3\phi _2\!\left [ \begin{matrix} q,q,q\\ dq,q^4/d
\end{matrix} ;q^2,q^2\right ]
=
\frac {(d,d/q^2;q^2)_\infty} {(d/q,d/q;q^2)_\infty}
   {}_2\phi _1\!\left [ \begin{matrix} q,q\\ dq
\end{matrix} ;q^2,\displaystyle\frac {d} {q}\right ]\\
-\frac {(d/q^2,q,q,q,d^2/q;q^2)_\infty} {(d/q,d/q,dq,q^2/d,d/q;q^2)_\infty}
   {}_3\phi _2\!\left [ \begin{matrix} d/q,d/q,d/q\\ d^2/q,d
\end{matrix} ;q^2,q^2\right ].
\end{multline*}
From here it is evident that
\begin{equation}
 \lim_{d\to0}  {}_3\phi _2\!\left [ \begin{matrix} q,q,q\\ dq,q^4/d
\end{matrix} ;q^2,q^2\right ]
=1.
\label{CK3}
\end{equation}
Indeed, the first term on the right-hand side of \eqref{CK3} converges
trivially to~$1$, while in the second term everywhere the substitution
$d=0$ is fine --- and thus produces a well-defined {\it finite} value
---,
except for the factor $(q^2/d,q^2)_\infty$ in the denominator.
However, as $d\to0$ (to be precise,
we take $d=q^2/D$ with $D\to-\infty$),
this term becomes unbounded,
whence the whole term tends to~0.
Thus, performing this limit, we get
\begin{multline*} \sum_{n=0}^\infty \frac {(1-q^{3n+2})\,(q^2;q^2)_n\,
(q,q,q;q)_n}
{(1-q^2)\,(q;q)_n\,(q^3,q^3,q^3;q^2)_n}q^{n+\binom n2}
\\=
\frac {(q^4,q^2,q^2,q^2;q^2)_\infty}
{(q,q^3,q^3,q^3;q^2)_\infty}
=
\frac {(1-q)^3\,(q^2,q^2,q^2,q^2;q^2)_\infty}
{(1-q^2)\,(q,q,q,q;q^2)_\infty},
\end{multline*}
as desired.
\end{proof}

\section{Alternative proofs of the identities of Guo and Liu}
\label{sec:4}

In this last section, we provide alternative proofs of
\eqref{eq:GL1} and \eqref{eq:GL2}, showing that they can be
obtained as special/limiting cases of a quadratic summation formula
due to Gasper and Rahman \cite{GaRaAB}.

\begin{proof}[Proof of \eqref{eq:GL1}]
We start with the quadratic summation formula
(cf.\ \cite[Eq.~(3.8.12)]{Gas04})
\begin{multline*} 
   \sum _{k=0} ^{\infty}
\frac { 1 - a {q^{3 k}}} {1-a}\cdot
 {{  \frac
    {({\let \over / \def\frac#1#2{#1 / #2} a, b, {q\over b}}; q) _{k}}
{({\let \over /
    \def\frac#1#2{#1 / #2} {q^2}, {{a {q^2}}\over b}, a b q}; {q^2}) _{k}}\cdot
\frac {({\let \over / \def\frac#1#2{#1 / #2}
    d, f, {{{a^2} q}\over {d f}}}; {q^2}) _{k}}  {({\let
    \over / \def\frac#1#2{#1 / #2} {{a q}\over d}, {{a q}\over f},
    {{d f}\over a}}; q) _{k}}
    }}\cdot{q^k}\\
+    \frac {(\let \over / \def\frac#1#2{#1 / #2} a q, {f\over a}, b, {f\over
     q} ;q) _\infty} {(\let \over / \def\frac#1#2{#1 / #2} {a\over f},
     {{f q}\over a}, {{a q}\over d}, {{d f}\over a} ;q) _\infty} \cdot
    \frac {(\let \over / \def\frac#1#2{#1 / #2} d, {{{a^2} q}\over {b d}},
     {{{a^2} q}\over {d f}}, {{f {q^2}}\over d}, {{d {f^2} q}\over
     {{a^2}}} ;{q^2}) _\infty} {(\let \over / \def\frac#1#2{#1 / #2}
     {{a {q^2}}\over b}, a b q, {{f q}\over {a b}}, {{b f}\over a},
     {{a {q^2}}\over {b f}} ;{q^2}) _\infty}  \\
\times
{} _{3} \phi _{2} \! \left [             \begin{matrix} \let \over / \def\frac#1#2{#1
     / #2} f, {{b f}\over a}, {{f q}\over {a b}}\\ \let \over /
     \def\frac#1#2{#1 / #2} {{f {q^2}}\over d}, {{d {f^2} q}\over
     {{a^2}}}\end{matrix} ;{q^2}, {\displaystyle {q^2}} \right ]
\\
 = \frac {(\let \over / \def\frac#1#2{#1 / #2} a q,
    {f\over a} ;q) _\infty} {(\let \over / \def\frac#1#2{#1 / #2} {{a q}\over
    d}, {{d f}\over a} ;q) _\infty} \cdot
   \frac {(\let \over / \def\frac#1#2{#1 / #2} {{a {q^2}}\over {b d}},
    {{a b q}\over d}, {{b d f}\over a}, {{d f q}\over {a b}} ;{q^2})
    _\infty} {(\let \over / \def\frac#1#2{#1 / #2} {{a {q^2}}\over b},
    a b q, {{b f}\over a}, {{f q}\over {a b}} ;{q^2}) _\infty}.
\end{multline*}
Now replace $f$ by $a^2q^{2N+1}/d$, with $N$ a positive integer.
The effect is that, because of the factor $(a^2q/df;q^2)_\infty$,
this kills off the second term on the left-hand side. In other
words, now this is indeed a genuine summation formula.
Now replace $q$ by $q^2$ and choose $a=b=q$. Then the above identity
reduces to
\begin{multline}\label{eq:red}
\sum_{k=0}^\infty
   \frac {1-q^{6 k+1}} {1-q}
   \cdot\frac{(d,q^{4
   N+4}/{d},q^{-4 N};
   q^4)_k}{(q^4,
   q^4,q^4;
   q^4)_k} \cdot\frac{(q,
   q,q;
   q^2)_k}{(q^3/{d},
   d q^{-4 N-1},q^{4 N+3};
   q^2)_k}\cdot q^{2 k}\\
=\frac{(q^3,
   q^{4 N+3}/{d};
   q^2)_{\infty}}{({q^3}/
   {d},q^{4 N+3};
   q^2)_{\infty}}
   \cdot\frac{({q^4}/{d},
   {q^4}/{d},q^{4
   N+4},q^{4 N+4};
   q^4)_{\infty}}{(q^4,
   q^4,{q^{4
   N+4}}/{d},{q^{4
   N+4}}/{d};
   q^4)_{\infty}}.
\end{multline}
Finally, we set $d=q^2$ and let $N\to\infty$.
Upon little simplification, we arrive at \eqref{eq:GL1}.
\end{proof}

\begin{proof}[Proof of \eqref{eq:GL2}]
We proceed in a similar manner. In \eqref{eq:red}, we set
$d=q^{-2N}$ with $N$ a positive integer. This yields the identity
\begin{multline*}
\sum_{k=0}^\infty
   \frac {1-q^{6 k+1}} {1-q}
   \cdot\frac{(q^{-2N},q^{6
   N+4},q^{-4 N};
   q^4)_k}{(q^4,
   q^4,q^4;
   q^4)_k}\cdot \frac{(q,
   q,q;
   q^2)_k}{({q^{2N+3}},
   q^{-6 N-1},q^{4 N+3};
   q^2)_k}\cdot q^{2 k}\\
=
\frac{(q^3;
   q^2)_N (q^{2N+4};
   q^4)_N^2}{(q^4
   ; q^4)_N^2
   (q^{4 N+3};
   q^2)_N}.
\end{multline*}
Letting $N\to\infty$, we then obtain \eqref{eq:GL2}.
\end{proof}

\medskip
{\bf Acknowledgment}. The authors would like to thank the anonymous referee for helpful comments.

\end{document}